\documentclass[12pt]{article}  
\usepackage{amsthm,amsmath}
\usepackage{amsfonts,amssymb}

\newtheorem{lemma}{Lemma}
\newtheorem{proposition}{Proposition}
\newtheorem{remark}{Remark}

\newtheorem{notation}{Notation}

\newcommand{\RM}{\mathbb{R}}

\newcommand{\CM}{\mathbb{C}}
\newcommand{\MM}{\,\mbox{\bf M}}

\newcommand{\cn}{\operatorname{cn}}

\newtheorem{thm}{Theorem}

\newtheorem{lem}{Lemma}
\newtheorem{prop}{Proposition}

\newtheorem{df}{Definition}

\usepackage{graphicx}

\begin{document}
\bibliographystyle{plain}
\title{An index theorem for the stability of periodic traveling waves
of KdV type.}
\author{Jared C. Bronski\thanks{Department of Mathematics, University of Illinois, 1409 W. Green St. Urbana, IL 61801 USA }\and Mathew A. Johnson\thanks{Department of Mathematics, Indiana University,831 East 3rd St, Bloomington, IN  47405 USA } \and Todd Kapitula\thanks{Department of Mathematics and Statistics, Calvin College, 1740 Knollcrest Circle SE, Grand Rapids, MI  49546 USA }}

\maketitle

\begin{abstract}
There has been a large amount of work aimed at understanding the
stability of nonlinear dispersive equations that support
solitary wave solutions\cite{GSS,BSS,Ben,Bona,Gard,PW1,PW2,Row,MIW1,MIW2}.
Much of this work relies on understanding detailed properties of the spectrum
of the operator obtained by linearizing the flow around the solitary wave.
These spectral properties, in turn, have important implications for the
long-time behavior of solutions to the corresponding partial differential
equation\cite{BJ,BP,DS,GSi1,GSi,KS,KZ,MMI,MMII,FMP,S,SW,SWI}.

In this paper we consider periodic solutions to equations of
Korteweg-Devries type. While the stability theory for periodic
waves has received much some attention\cite{AP,BD,APBS,BR,GH1,GH2,HK,DK} the
theory is much less developed than the analogous theory for solitary wave
stability, and appears to be mathematically richer. We prove an
index theorem  giving
an exact count of the number of unstable eigenvalues of the
linearized operator in terms of the number of zeros of the derivative
of the traveling wave profile together with geometric information about a
certain map between the constants of integration of the ordinary 
differential equation and the conserved quantities of the partial 
differential equation.

 This index can be regarded as a generalization of 
both the Sturm
oscillation theorem and the classical stability theory for solitary wave
solutions for equations of Korteweg-de Vries type.
In the case of a polynomial nonlinearity this index, together with a
related one introduced earlier by Bronski and Johnson,
can be expressed in terms of derivatives of period integrals on a Riemann
surface. Since these period integrals satisfy a Picard-Fuchs equation
these derivatives can be expressed in terms of the integrals themselves,
leading to an expression in terms of various moments of the solution.
We conclude with some illustrative examples.

\end{abstract}

\section{Introduction and  Preliminaries}

In the area of nonlinear dispersive waves the question of stability
is an important one, as it determines what states one is likely to observe
in practice. In this paper we consider periodic traveling wave solutions
to equations of KdV type:
\begin{equation}
u_t + u_{xxx} + (f(u))_x = 0\label{gkdv}
\end{equation}
where $f$ is assumed to be $C^2.$  The main goal of this paper is to determine
an index expressible in terms of the conserved quantities restricted to the manifold
of periodic traveling wave solutions which yields sufficient information for the
orbital and spectral stability of the underlying wave.  Our results are geometric
in nature as they relate to Jacobians of various maps which relate naturally
to the structure of the underlying linearized operator, and will be most explicit when $f$ is a polynomial.
Note that when $f$ is a polynomial of
degree at most three the above is integrable via the inverse scattering
transform, but this is not the case for polynomials of
higher degree.

Assuming a traveling wave of the form $u = u(x - ct)$ one is
immediately led to the following nonlinear oscillator
equation
\begin{equation}
\frac{u_x^2}{2} - \frac{c u^2}{2} + a u + F(u) = E
\label{eqn:ODE}
\end{equation}
where $F$ is the antiderivative of the nonlinearity $f.$ Thus the
periodic waves  depend on three parameters $a,E,c$ together with
a fourth constant of integration $x_0$ corresponding to the
translation mode which can be modded out. Thus when we speak of a
three parameter family of solutions we will be refering to $a,E,c$.
On open sets in ${\mathbb{R}}^3 = (a,E,c)$ the solution to the
above is periodic. The boundary of this locus of points where
the discriminant vanishes, which includes the constant solutions and
the solitary waves.

The constants $a$ and $c$ admit a variational interpretation: defining the
Hamiltonian function
\[
H = \int \frac{u_x^2}{2} - F(u) dx
\]
so that the Korteweg-DeVries equation can be written
\[
u_t = \frac{\partial}{\partial x} \delta H(u)
\]
then the traveling waves are critical points of the augmented Hamiltonian
functional
\[
\delta (H - a M - c P/2) = 0
\]
and thus represent critical points of the Hamiltonian under the constraint
of fixed mass and momentum, with the quantities $a,c$ representing Lagrange
multipliers enforcing the constraints of fixed mass and momentum respectively.

We also note the connection with the underlying classical mechanics.
The traveling wave ordinary differential equation is (after one integration)
Hamiltonian and integrable. The classical action for this ordinary differential 
equation is
\[
K = \oint p dq = \sqrt{2} \oint \sqrt{E + a u + \frac{cu^2}{2} - F(u)} du.
\]
The classical action provides a generating function for the
conserved quantities of the traveling waves: specifically the classical
action satisfies the following relationships
\begin{align}
& T = \frac{\partial K}{\partial E} \\
& M = \frac{\partial K}{\partial a} \\
& P = 2 \frac{\partial K}{\partial c}.
\end{align}

In this paper, we are interested in both the spectral and orbital (nonlinear) stability
of spatially periodic traveling wave solutions of \eqref{gkdv}.  The spectral stability
problems has been recently considered \cite{BD,BrJ,HK} in which the authors considered
stability to localized perturbations.
In this case, the linearized stability takes the form
\begin{equation}
{\bf J} {\mathcal L}\mu v = \mu v,\label{eqn:operator}
\end{equation}
where $\mathcal{L}=-\partial_x^2-f'(u)+c$ is a differential operator with periodic coefficients
considered on the real Hilbert space $L^2(\RM)$ and ${\bf J}=\partial_x$ gives the Hamiltonian 
structure.  This is the standard form for the stability problem
for solutions to equations with a Hamiltonian structure, although it must
be emphasized that in the KdV case ${\bf J}=\partial_x$ has a non-trivial kernel (spanned by $1$) which
complicates matters somewhat.  

\begin{notation}
Throughout this paper $T$ will denote the period of the underlying 
periodic traveling wave. We will let 
\[
{\mathbb T}_k= {\mathbb R}/(k T {\mathbb Z})
\]
 be the torus of length $kT$ and $L^2({\mathbb T}_k)$, the Hilbert 
space of square integrable $kT$ periodic functions. We will frequently 
restrict to the subspace of mean zero functions. Following the 
notation of Deconinck and Kapitula we will denote this
subspace by $H_1:$
\[
H_1 = \{ \phi \in L^2({\mathbb T}_k) ~| ~\langle1,\phi\rangle = 0 \}. 
\]

In this paper we give a geometric construction of the spectrum of the
linearized operator about a periodic traveling wave solution.
This construction depends on the Jacobian determinants of various maps.
Throughout this paper we will use the following Poisson bracket style notation
for Jacobian determinants
\[
\{F,G\}_{x,y} = \left|\begin{array}{cc}F_x & F_y \\ G_x & G_y \end{array}\right|
\]
with the analogous notation for larger Jacobian determinants:
\[
\{F,G,H\}_{x,y,z} = \left|\begin{array}{ccc}F_x & F_y & F_z\\ G_x & G_y & G_z \\
H_x & H_y & H_z\end{array}\right|.
\]
\end{notation}

We will also define the following eigenvalue counts:

\begin{df}
Given the linearized operator $\partial_x\mathcal{L}$ acting on $L^2_{\rm per}({\mathbb T}_k)$
we define $k_{\RM}$ to be the number of real strictly positive eigenvalues, $k_{\CM}$ to be the number of 
complex-valued eigenvalues with strictly positive real part, and $k_{\mathbb{I}}^-$ to be the number of purely
imaginary (non-zero) eigenvalues with negative Krein signature - in other words
eigenvalues such that the corresponding eigenfunctions $w$ satisfy
\[
\langle w, {\mathcal L} w\rangle <0.
\]
\end{df}

It is worth making a few remarks on this definition. The Krein signature is an important geometric quantity
associated with eigenvalue problems having a Hamiltonian structure\cite{YS}, and is associated with the
sense of transversality of the root of the eigenvalue relation. It is a fundamental result that two 
eigenvalues of like Krein signature collide they will remain on the axis, while if two eigenvalues of 
opposite Krein signature collide they will (generically) leave the imaginary axis. It can be shown that  a band of spectrum on the imaginary axis with multiplicity one will have eigenvalues of positive 
Krein signature, while a band of spectrum of multiplicity three will have two eigenvalues of positive 
Krein signature and one eigenvalue of negative Krein signature. It is not hard to check that outside of 
a sufficiently large ball in the spectral plane the spectrum lies on the imaginary axis and 
has multiplicity one: it follows from this that the non-imaginary eigenvalues and imaginary 
eigenvalues of negative Krein signature (being the roots of an analytic function) must be finite 
in number. Also notice that by symmetry of the spectrum
about the real and imaginary axes, the quantities $k_{\CM}$ and $k_{\mathbb{I}}^-$ are necessarily even, while the
quantity $k_{\RM}$ has no definite parity.  Our immediate goal is to establish the following index theorem, which is the main
result of the paper:

\begin{thm}\label{thm:index}
Suppose $u$ is a periodic solution to \eqref{eqn:ODE} and
$T,M,P$ be the period, mass, and momentum of this solution
considered as functions of the
parameters $(a,E,c)$:
 \begin{align}
& T = \oint \frac{du}{\sqrt{2(E + a u + \frac{c}{2} u^2 - F(u))}} \label{period}\\
& M = \oint \frac{u du}{\sqrt{2(E + a u + \frac{c}{2} u^2 - F(u))}} \label{mass}\\
& P = \oint \frac{u^2 du}{\sqrt{2(E + a u + \frac{c}{2} u^2 - F(u))}}.\label{momentum}
\end{align}
Also let $k_{\mathbb R},k_{\mathbb C},k_{\mathbb I}^-$ be the number of real, complex and imaginary eigenvalues of 
negative Krein signature of  $\partial_x {\mathcal L}$
on $L^2_{\rm per}({\mathbb T}_k)$ as defined above. Finally suppose that none of $T_E, \{T,M\}_{a,E},
\{T,M,P\}_{a,E,c}$ are zero. Then we have the following equality
\begin{equation}
k_{\mathbb I}^-+k_{\mathbb{R}}+k_{\mathbb C}=2k-1 + \left\{\begin{array}{c} 0 ~~T_E>0 \\ 1 ~~T_E<0 \end{array}\right\} - \left\{\begin{array}{c} 0 ~~\frac{\left\{T,M\right\}_{a,E}}{T_E}>0 \\ 1 ~~\frac{\left\{T,M\right\}_{a,E}}{T_E}<0
\end{array}\right\}
-  \left\{\begin{array}{c} 0 ~~\frac{\left\{T,M,P\right\}_{a,E,c}}{\left\{T,M\right\}_{a,E}}<0 \\ 1 ~~\frac{\left\{T,M,P\right\}_{a,E,c}}{\left\{T,M\right\}_{a,E}}>0
\end{array}
\right\}\label{index}
\end{equation}
\end{thm}

The count on the left hand side of \eqref{index} clearly gives information concerning
the spectral stability with perturbations in $L^2_{\rm per}({\mathbb T}_k)$.  In particular,
if the count is zero then one can conclude the underlying periodic wave is spectrally
stable to perturbations in $L^2_{\rm per}({\mathbb T}_k)$.  Moreover,
if the count is odd one is guaranteed the existence of a real eigenvalue since the quantities
$k_\CM$ and $k_\mathbb{I}^-$ are necessarily even.
What is less clear is the relation of this count to the orbital stability of such a solution
in $L^2_{\rm per}({\mathbb T}_k)$.  We will discuss this relationship in detail in the next section.
The right hand side of \eqref{index} relates to geometric information concerning the underlying
periodic traveling wave profile, and the equality provides a direct relationship between the conserved
quantities of the governing PDE flow restricted to the manifold of periodic traveling wave solutions to the
stability of the underlying wave itself.

In the case $F$ is polynomial the integrals \eqref{period}, \eqref{mass}, and \eqref{momentum} are Abelian integrals on a
Riemann surface and the above expressions can be greatly simplified. For
instance for the case of the Korteweg-de Vries equation the quantities
$T_E,\{T,M\}_{a,E},\{T,M,P\}_{a,E,c}$ are homogeneous polynomials of degrees
one, two and three resepctively in $T,M$, while for the modifed
Korteweg-de Vries equation
they are homogeneous polynomials in $T$ and $P$. In general for a polynomial
nonlinearity they are homogeneous polynomials of degree one, two and three in
some finite number of moments of the solution $I_k =\oint \frac{u^k du}{\sqrt{E + a u + \frac{c}{2} u^2 - F(u)}}.$
Thus, in the polynomial nonlinearity case Theorem \ref{thm:index} yields sufficient information
for the stability of a periodic traveling wave solution in terms of a finite number
of moments of the solution itself.

\section{Background and Main Results}

The study of eigenvalues of operators of the form \eqref{eqn:operator} has a
long history.  Eigenvalue problems of exactly this form arise in the study of the stability
of solitary wave solutions to equations of Korteweg-de Vries type. The basic
observation is that, if ${\mathcal L}$ were positive definite the spectrum
of ${\bf J}{\mathcal L}$ would necessarily be purely imaginary, since this
operator is skew-adjoint under the modified inner product
$\langle\langle u,v\rangle\rangle=\langle {\mathcal L}^{1/2}u,{\mathcal L}^{1/2}v\rangle$, with
$\langle \cdot,\cdot \rangle$ the standard inner product on $L^2({\mathbb R})$.
While in the case of nonlinear dispersive waves  ${\mathcal L}$ is never
positive definite due to the presence of symmetries one can count the number
of eigenvalues off of the imaginary axis in terms of the dimensions of the
kernel and the negative definite subspace of ${\mathcal L}.$

In the case of periodic solutions to the Korteweg-DeVries equation the
best results of this type that we are aware of are due to Haragus and
Kapitula\cite{HK} and Deconinck and Kapitula\cite{DK}.
Kapitula and Deconinck give the following construction:
consider the spectral problem \eqref{eqn:operator} with the linearized operator ${\bf J}\mathcal{L}$
acting on the real Hilbert space $L^2({\mathbb T}_k)$, and let $k_\RM$, $k_\CM$,
and $k_\mathbb{I}^-$ be defined as before.
Then one has the count
 \[
k_{\mathbb I}^-+k_{\mathbb{R}}+k_{\mathbb C}= n({\mathcal L}|_{H_1}) - n(D).
\]
where $n(\cdot)$ denotes the dimension of the negative definite subspace of the appropriate operator
acting on $L^2({\mathbb T}_k)$,  and
$D$ is a symmetric matrix whose entries are given by
\[
D_{i,j}=\left<y_i,\mathcal{L}|_{H_1}y_j\right>
\]
where $\{y_i\}$ is any basis for the generalized eigenspace of $\partial_x\mathcal{L}$ such that
\[
{\bf J}\mathcal{L}|_{H_1}\;{\rm span}\{y_i\}={\rm ker}(\mathcal{L}|_{H_1}).
\]
The importance of this formula is the following: By using the results of
\cite{GSS}, it is known that a sufficient condition for the orbital stability
of a periodic traveling wave solution of \eqref{gkdv} is given
by $n(\mathcal{L}|_{H_1})=n(D)$.  Thus, if one is able to
prove that
\[
k_{\mathbb I}^-+k_{\mathbb{R}}+k_{\mathbb C}=0,
\]
one can immediately conclude orbital stability in $L^2({\mathbb T}_k)$.
It follows that spectral stability can be upgraded to the orbital stability if there
are no purely imaginary eigenvalues of negative Krein signature.  Moreover, the count
clearly gives information concerning the spectral stability of the underlying periodic
wave.  In particular, a necessary condition for the spectral stability of such a solution
in $L^2({\mathbb T}_k)$ is for the difference $n(\mathcal{L}|_{H_1})-n(D)$
to be even: more will be said on this later.

In another paper Bronski and  Johnson \cite{BrJ} considered the analogous spectral stability problem
to localized perturbations from the
point of view of Whitham modulation theory.  Bronski and Johnson gave a
normal form calculation for the spectral problem in a neighborhood
of the origin in the spectral plane, which amounts to studying the spectral stability
of a periodic traveling wave solution of \eqref{gkdv} to long-wavelength perturbations: so called
modulational instability.  It was found that the presence of such an instability could be detected
by computing various Jacobians of maps from the conserved quantities of the gKdV flow to the
parameter space $(a,E,c)$ used to parameterize the periodic traveling waves.
By deriving an asymptotic expansion of the periodic Evans function
\[
D(\mu,e^{i\kappa})=\det\left(\MM(\mu)-e^{i\kappa}{\bf I}\right)
\]
in a neighborhood of $(\mu,\kappa)=(0,0)$, where $\MM(\mu)$ is the monodromy
matrix associated with third order ODE \eqref{eqn:operator} and ${\bf I}$ is the
three-by-three identity matrix, it was found that the spectrum of the operator ${\bf J}\mathcal{L}$
in a neighborhood of the origin is determined by the modulational instability index
\begin{equation}
\Delta_{\rm MI}=\frac{1}{2}\left(\{T,P\}_{E,c}+2\{M,P\}_{E,a}\right)^3-3\left(\frac{3}{2}\{T,M,P\}_{a,E,c}\right)^2.\label{MI}
\end{equation}
In particular, it was found that if $\Delta_{\rm MI}>0$ then the spectrum locally consists of a
symmetric interval on the imaginary axis with multiplicity three, while if $\Delta_{\rm MI}<0$
the spectrum locally consists of a symmetric interval of the imaginary axis with multiplicity one, along with
two branches which, to leading order, bifurcate from the origin along straight lines with non-zero slope.
Such information is important if one wishes to consider spectral or nonlinear stability
to perturbations whose fundamental period is an integer multiple of that of the underlying wave.

A similar geometric construction was later found useful by Johnson \cite{MJ1} to prove that such a solution
with fundamental period $T$ is orbitally stable in $L^2({\mathbb T}_k)$ if
the quantities $T_E$, $\{T,M\}_{a,E}$, and $\{T,M,P\}_{a,E,c}$ are all positive.  Notice
that for this sign pattern Theorem \ref{thm:index} also implies orbital stability in
$L^2({\mathbb T}_k)$, and hence the index theorem can be regarded as an extension
of the theorem of Johnson.  Both the calculations of Bronski and Johnson and of Johnson
required a detailed understanding of the structure of the kernel and generalized kernels
of the linear operators $\partial_x {\mathcal L}$ and ${\mathcal L\partial_x}$ with periodic boundary conditions.  These kernels
were constructed by taking infinitesimal variations in each of the defining parameters
$a$, $E$, and $c$ as well as the translation mode: this construction will be reviewed in the
Proposition \ref{prop:kernel}.  This information on the structure of the null-spaces, together with some additional work,
will allow us to prove our main result.

The operators we consider in this paper are non-self-adjoint and
the null-spaces typically have a non-trivial Jordan structure. We adopt the
following notation:
\begin{notation}
Given an operator ${\mathcal A}$ acting on $L^2({\mathbb T}_k)$ for some $k\in\mathbb{N}$,
we define the $k^{\rm th}$ generalized kernel as follows
\[
{\rm g\!\!-\!\!ker}_k({\mathcal A})= {\rm ker}(A^{k+1})/{\rm ker}(A^{k}).
\]
Thus ${\rm g\!\!-\!\!ker}_0({\mathcal A})=\ker(A)$ is the usual kernel, $A({\rm g\!\!-\!\!ker}_1(A))=\ker(A)$, and so on.
\end{notation}
We begin by stating a preliminary lemma regarding the Jordan structure
of $\partial_x {\mathcal L}$

\begin{prop}\label{prop:kernel}
Given any $k\in\mathbb{N}$, one generically has
$\dim(\ker({\mathcal L}))=1$, $\dim(\ker({\partial_x}{\mathcal L}))=2$,
$\dim({\rm g\!\!-\!\!ker}_1({\partial_x}{\mathcal L}))=1,$ and ${\rm g\!\!-\!\!ker}_j({\partial_x}{\mathcal L})=\emptyset$
for $j\ge 2$.  In particular, we have the following result:
\begin{itemize}
\item{If $T_E\neq 0$ then ${\rm ker}{\mathcal L} = {\rm span}\{u_x\}. $ If $T_E=0$
then  ${\rm ker}{\mathcal L} = {\rm span}\{u_x,u_E\}. $}
\item{If $T_E$ and $T_a$ do not simultaneously vanish then
\begin{align}
&{\rm ker}(\partial_x {\mathcal L})={\rm span}\left\{ u_x,\left|\begin{array}{cc}u_E & T_E \\ u_a & T_a \end{array}\right| \right\} \\
&{\rm ker}({\mathcal L}\partial_x )={\rm span}\{ 1,u \}
\end{align}
}
\item{If  $T_E$ and $T_a$ simultaneously vanish then
 \begin{align}
&{\rm ker}(\partial_x {\mathcal L})={\rm span}\{u_x,u_a,u_E\}\\
&{\rm ker}({\mathcal L}\partial_x)={\rm span}\left\{1,u,\int_0^x \!\!u_E~dx \right\}.
\end{align}
Since the defining ordinary differential equation is third order the
kernel cannot be more than three dimensional.
}
\item{
If $\{T,M\}_{a,E}\neq 0$  then
\begin{align}
&{\rm g\!\!-\!\!ker}_1(\partial_x{\mathcal L})=
{\rm span} \left\{\left|\begin{array}{ccc}u_E & T_E & M_E\\ u_a & T_a & M_a \\ u_c & T_c & M_c \end{array}\right|\right\} \\
&{\rm g\!\!-\!\!ker}_1({\mathcal L}\partial_x)={\rm span} \left\{ \int_0^x\left|\begin{array}{ccc}u_E & T_E & M_E\\ u_a & T_a & M_a \\ u_c & T_c & M_c \end{array}\right|\right\}
\end{align}
The generalized kernels ${\rm g\!\!-\!\!ker}_1$ are one dimensional unless
$\{T,M\}_{a,E}$ and $\{T,P\}_{a,E}$ vanish simultaneously.
}
\item{
Assuming $\{T,M\}_{a,E}\neq 0$ the subsequent generalized kernels
${\rm g\!\!-\!\!ker}_k$ ($k\ge 2$) are empty as long as $\{T,M,P\}_{E,a,c} \neq 0$
}
\end{itemize}
\end{prop}

\begin{proof}
This follows from the observation that the derivatives of the wave profile
$u$ with respect to the parameters $a,E,c$ satisfy the following equations
\begin{align}
& {\mathcal L}u_x = 0 \\
& {\mathcal L}u_E = 0 \\
& {\mathcal L}u_a = -1 ~~~~~(=-\frac{\delta M}{\delta u}) \\
& {\mathcal L}u_c = -u ~~~~~(=-\frac{\delta P}{\delta u})
\end{align}
reflecting the fact that the constants $(a,c)$ arise as Lagrange
multipliers to enforce the mass and momentum constraints. In the above
equality ${\mathcal L}$ denotes the formal operator without consideration
for boundary conditions. In order to find elements of the kernel
one must impose periodic boundary conditions.  It is not hard to see that
$u_x$ is periodic while derivatives with respect to the quantities
are not periodic - since the period $T$ depends on $(E,a,c)$ ``secular'' terms
(in the sense of multiple scale perturbation theory) arise: in particular
one sees that the change across a period is proportional to derivatives of the
period:
\[
 \left(\begin{array}{c} u_E(T) \\ u_{xE}(T) \\ u_{xxE}(T) \\ \vdots \end{array}\right) -
 \left(\begin{array}{c} u_E(0) \\ u_{xE}(0) \\ u_{xxE}(0) \\ \vdots \end{array}\right) = T_E  \left(\begin{array}{c} u_{xE}(0) \\ u_{xxE}(0) \\ u_{xxxE}(0) \\ \vdots \end{array}\right)
\]
with similar expressions for the change in the $u_a,u_c$ across a period. Thus the quantity
\[
\phi_1(x;a,c,e)=\left| \begin{array}{cc} u_E & T_E \\ u_a & T_a \end{array}\right|
\]
is periodic and satisfies ${\mathcal L}\phi_1 = T_E.$ Similarly the quantity
\[
\phi_2(x;E,a,c)=\left| \begin{array}{ccc} u_E & T_E & M_E\\ u_a & T_a & M_a \\
u_c & T_c & M_c\end{array}\right|
\]
is by construction periodic and satisfies  ${\mathcal L}\phi_2 = \{T,M\}_{E,c}- \{T,M\}_{E,a} u,$ and thus $\partial_x{\mathcal L}\phi_2 = \{T,M\}_{E,a} u_x \in {\rm ker}(\partial_x {\mathcal L}).$ Note that while $\phi_1$ is essentially uniquely
determined $\phi_2$ is only determined up to an element of the kernel.
Here we have chosen
to make $\phi_2$ have mean zero since this is the convention
required in the work of Deconinck and Kapitula. More will be said on this choice
later.

The rest of the calculation follows in straightforward way from calculations
of this sort. For instance the existence of a second element of the generalized
kernel is equivalent to the solvability of
\[
\partial_x {\mathcal L} = \left|\begin{array}{cc}u_E & T_E \\ u_a & T_a \end{array}\right|.
\]
By the Fredholm alternative ${\rm ran}(\partial_x {\mathcal L}) = {\rm ker}({\mathcal L} \partial_x)^\perp$ and thus the above is solvable if only if
\begin{align}
& \langle 1, \left|\begin{array}{cc}u_E & T_E \\ u_a & T_a \end{array}\right| \rangle = \{T,M\}_{a,E} = 0 \\
& \langle u, \left|\begin{array}{cc}u_E & T_E \\ u_a & T_a \end{array}\right| \rangle = \{T,P\}_{a,E} = 0
\end{align}
The rest of the claims follow similarly. In the case that the genericity
conditions do not hold we do not attempt to compute the Jordan form,
but we do remark that the algebraic multiplicity of the zero eigenvalue
must jump from three to at least five, and is necessarily odd.
\end{proof}

In essence the above proposition shows that the elements of the kernel of
$\partial_x {\mathcal L}$ are given by  elements of
the tangent space to the (two-dimensional) manifold of solutions
{\em of fixed period at
fixed wavespeed}, while the element of the first generalized kernel
is given by a vector in the tangent space to the (three-dimensional)  manifold of solutions
{\em of fixed period} with no restrictions on wavespeed. As one might expect
all of the geometric information on independence in the above
proposition can be expressed in terms of various Jacobians.
The next fact we note is that the signs of certain of these quantities
conveys geometric information about the various operators.

\begin{lemma}
Let $n({\mathcal L})$ be the dimension of the negative definite subspace of
$\mathcal L$ as an operator on $L^2({\mathbb T}_k)$ with periodic boundary
conditions. Then
\[
n({\mathcal L}) = \left\{\begin{array}{c}2k-1 ~~~~~~T_E>0 \\ 2k ~~~~~~T_E<0 \end{array}\right.
\]
\end{lemma}
\begin{proof}
As noted above $u_x \in \ker({\mathcal L})$ and thus zero is a band-edge
of ${\mathcal L}$. From the Sturm oscillation theorem it is clear that either
 $n({\mathcal L})=n_{u_x}-1 $ (if zero is an upper band-edge) or  $n({\mathcal L})=n_{u_x}$ if zero is a lower band-edge. The Floquet discriminant $k(\mu)$
has positive slope at an upper band-edge and negative slope at a lower
band-edge (and vanishes at a double point), and thus serves to distinguish the
two cases.  It can be shown (see appendix) that the sign of $k'(\mu)$ is
equal to the sign of $T_E$ and thus the result follows. Note that the
sign of the slope of the Floquet discriminant has an interpretation as
the Krein signature of the eigenvalues in the band.
\end{proof}

This lemma implies that the  vanishing of $T_E$ signals an
eigenvalue of ${\mathcal L}$ passing through the origin and a change in
the dimension of $n({\mathcal L})$. The next proposition gives a similar
interpretation for  $n({\mathcal L}|_{H_1})$.

\begin{proposition}
Let $H_1 (={\rm ran}(\partial_x{\mathcal L})$ denote the space of mean-zero, 
$T$ periodic functions, and
$n({\mathcal L|_{H_1}})$ denote the dimension of the
negative-definite subspace of the restriction of the operator ${\mathcal L}$ 
to this subspace.
Assume that $\{T,M\}_{a,E}\neq 0$ and $T_E and \{T,M\}_{a,E}$ never vanish
simultaneously. Then we have the equality
\[
n({\mathcal L}|_{H_1}) = n({\mathcal L}) - \left\{ \begin{array}{c}0 ~~~~~~T_E \{T,M\}_{a,E} >0  \\
1 ~~~~~~T_E \{T,M\}_{a,E} <0  \end{array}\right\}
\]
\label{prop:count2}
\end{proposition}

\begin{proof}
Since we are restricting to a codimension one subspace the
Courant minimax principle immediately implies that we have
either   $n({\mathcal L}|_{H_1})= n({\mathcal L})$ or
$n({\mathcal L}|_{H_1})= n({\mathcal L})-1$.

Next note that when $\{T,M\}_{a,E}=0$ then the function $\{u,T\}_{a,E}$
has mean zero and lies in $H_1$. Further one has that
\[
{\mathcal L}\{u,T\}_{a,E} = -T_E
\] or equivalently
\[
{\mathcal L}|_{H_1}\{u,T\}_{a,E} = 0.
\]
Therefore the vanishing of  $\{T,M\}_{a,E}$ signals a change in the dimension
of $n({\mathcal L}|_{H_1})$. A local perturbation analysis shows that near
a zero of $\{T,M\}_{a,E}$ we have that there is an eigenvalue
of ${\mathcal L}|_{H_1}$ which is given by
\[
\lambda = - \frac{T_E \{T,M\}_{a,E}}{\Vert \phi_1\Vert^2} + o(\{T,M\}_{a,E})
\]
Thus as long as $ \{T,M\}_{a,E}$ vanishes for some parameter value
the above count is correct. The case that $\{T,M\}_{a,E}$ is non-vanishing
will be handled later.  
\end{proof}

\begin{remark}
The last two results show that geometric quantities associated to the
classical mechanics of the traveling waves contain information about
changes in the nature of the spectrum of the linearized problem. Specifically:
\begin{itemize}
\item Vanishing of $T_E= K_{EE}$ signals  a change in the dimension of $n({\mathcal L})$, the negative definite subspace of $\mathcal L.$
\item Vanishing of $ \{T,M\}_{a,E} = -\left|\begin{array}{cc}K_{EE} & K_{Ea} \\ K_{aE} & K_{aa}\end{array}\right|$ signals a change in the $n({\mathcal L}|_{{\rm Ran}(\partial_x)})$,  the negative definite subspace of ${\mathcal L}|_{{\rm Ran}(\partial_x)}$.
\item   Vanishing of $ \{T,M,P\}_{a,E,c}$ signals a change in the length
of the Jordan chain of $\partial_x {\mathcal L}$.
\end{itemize}
\end{remark}

Finally, to conclude the proof of Theorem \ref{thm:index}, we must calculate $n(D)$.  This is the content
of the following lemma.

\begin{lem}\label{D}
Under the assumptions of Theorem \ref{thm:index}, one has that $D\in\RM$ with
\[
D=-\{T,M\}_{a,E}\{T,M,P\}_{a,E,c}.
\]
Thus, $n(D)$ is either $0$ or $1$ depending if $\{T,M\}_{a,E}\{T,M,P\}_{a,E,c}$ is negative
or positive, respectively.
\end{lem}

\begin{proof}
Under the assumptions of Theorem \ref{thm:index}, we know that ${\rm ker}(\mathcal{L})={\rm span}\{u_x\}$
and
\[
\partial_x\mathcal{L}|_{H_1}\left| \begin{array}{ccc} u_E & T_E & M_E\\ u_a & T_a & M_a \\
u_c & T_c & M_c\end{array}\right|=-\{T,M\}_{a,E}u_x
\]
from Proposition \ref{prop:kernel}.  It follows that the matrix $D$ is 
a real number in this case, with value equal to
\[
D=\left<\left| \begin{array}{ccc} u_E & T_E & M_E\\ u_a & T_a & M_a \\
u_c & T_c & M_c\end{array}\right|,\mathcal{L}\left| \begin{array}{ccc} u_E & T_E & M_E\\ u_a & T_a & M_a \\
u_c & T_c & M_c\end{array}\right|\right>=-\{T,M\}_{a,E}\{T,M,P\}_{a,E,c}
\]
as claimed.

To complete the proof of proposition \ref{prop:count2} note that the 
results of Bronski and Johnson imply the following identity:
\[
k_{\mathbb{R}} \equiv \left\{\begin{array}{cc} 0 {\rm mod}~ 2 & \{T,M,P\}_{a,E,c} >0 \\  1 {\rm mod}~ 2 & \{T,M,P\}_{a,E,c} <0\end{array} \right.
\]
Since $k_{\mathbb I}^-$ and $k_{\mathbb C}$ are even they do not change the count modulo two. In the case $\{T,M\}_{a,E}$ is non-vanishing the count is determined 
to within one, and is thus the count is exact if one knows the parity. 
Applying the result of Bronski and Johnson thus determines the count.
\end{proof}

\begin{remark}
It should be noted that the quantity $D$ computed in Lemma \ref{D} also arose naturally
in \cite{MJ1} when considering orbital stability
of periodic traveling wave solutions of \eqref{gkdv} to perturbations with
the same periodic structure.  There, the negativity of $D$ was necessary in order to prove the quadratic form
induced by $\mathcal{L}$ acting on $L^2({\mathbb T}_k)$ was positive definite
on an appropriate subspace.  As the methods therein are based on classical energy functional calculations,
such a requirement was necessary to classify the periodic traveling wave as a local minimizer of the Hamiltonian
subject to the momentum and mass constraints.
\end{remark}

The proof of Theorem \ref{thm:index} is now complete.  Notice that Theorem \ref{thm:index} gives a sufficient
requirement for a spatially periodic traveling wave of \eqref{gkdv} to be orbitally stable
in $L^2({\mathbb T}_k)$ for any $k\in\mathbb{N}$ and any sufficiently smooth nonlinearity
$f$.  In the next two sections, we analyze Theorem \ref{thm:index} in the case of a power-nonlinearity
by using complex analytic methods to reduce the expression for the Jacobians involved in \eqref{index}
in terms of moments of the underlying wave itself.  This has the obvious advantage of being more susceptible
to numerical experiments as one no longer has to numerically differentiate approximate solutions
with respect to the parameters $a$, $E$, and $c$.  We will also discuss the computation of $\Delta_{\rm MI}$
for power-law nonlinearities.  In particular, we will prove a new theorem in the case of the focusing
and defocusing MKdV which relates the modulational stability of a spatially periodic traveling wave
to the number of distinct families of periodic solutions existing for the given parameter values.

\section{Polynomial Nonlinearities and the Picard-Fuchs System}

One major simplification of this theory occurs when the nonlinearity $f(u)$
is polynomial. In this case the fundamental quantities $(T,M,P)$
are given by Abelian integrals of the first, second or third kind
on a Riemann surface. While we cannot give a detailed exposition of
this theory here the basics are very straightforward. Suppose that
$P(u)=a_0 + a_1 u + \ldots a_n u^n$ is a polynomial of degree $n$.
If the polynomial is of degree $2g+2$ or $2g+1$ then the quantity
$\frac{u^kdu}{\sqrt{P(u)}}$ is an Abelian differential on a Riemann
surface of genus $g$. If we define
the $k^{\rm th}$ moment $\mu_k$ of the solution $u(x)$ as follows:
\[
\mu_k = \int_0^T u^k(x) dx = \oint_\gamma \frac{u^k du}{\sqrt{P(u)}}
\]
then one obviously has
\[
\frac{d \mu_k}{d a_j} = \frac{d \mu_j}{d a_k} =\oint_\gamma \frac{u^{k+j} du}{P^{\frac{3}{2}}(u)} =: I_{k+j}
\]
for any loop $\gamma$ in the correct homotopy class. For our purposes we
are interested in branch cuts on the real axis though none of what will
be said in this section assumes this.
In the context of the stability problem one only needs $I_0\ldots I_4,$
since $T_E,T_a, \ldots P_c$ can all be expressed in terms of these five
quantities, but the theory requires that one consider all such moments.
The main observation is that the above integrals $\{I_k\}$ are
again Abelian integrals and thus can be expressed in terms of $\{\mu_k\}.$

In practice the simplest way to do this is to use the identity
 \[
\mu_m = \oint \frac{u^m P(u) du}{P^{\frac{3}{2}}(u)} = \sum_{j=0}^n a_j I_{j+m}
\]
for $m\in (0..n-1)$ and
\begin{align}
& \oint \frac{u^m P'(u) du}{P^{\frac{3}{2}}(u)} = 2 m \oint \frac{u^{m-1}du}{\sqrt{P(u)}} = 2 n \mu_{m-1}\\
& \sum_{j=0}^n j a_j I_{j+m-1} =  2 n \mu_{m-1}
\end{align}
for $m \in \{0,1,...,n\}$. This gives a linear system of $2n-1$ equations in
 $2n-1$ unknowns $\{I_k\}_{k=0}^{2n-1}$:
\[
\left(\begin{array}{ccccccc}a_0 & a_1 & \ldots & a_n & 0 & 0 & \ldots  \\
0 & a_0 &  a_1 & \ldots & a_n & 0 & \ldots \\
\vdots & \ddots & \ddots & \ddots &  \ddots & \ddots & \ddots\\
0 & \ldots & 0 & a_0 & a_1 & \ldots & a_n \\
a_1 & 2 a_2 & \ldots & n a_n & 0 & 0 & \ldots \\
0 & a_1 & 2 a_2 & \ldots & n a_n & 0 & \ldots \\
\vdots & \ddots & \ddots & \ddots &  \ddots & \ddots & \ddots\\
0 & \ldots & 0 & a_1 & 2a_2 & \ldots & na_n
\end{array}\right)\left(\begin{array}{c} I_0 \\ I_1 \\ I_2 \\ I_3 \\ \vdots \\ I_{2n-3} \\ I_{2n-2} \end{array}\right) =\left(\begin{array}{c} \mu_0 \\ \mu_1 \\ \mu_2 \\ \mu_3 \\ \vdots \\ \mu_{n-1} \\ \mu_n \end{array}\right)
\label{eqn:system}
\]
The matrix  which arises in the above linear systems is the Sylvester matrix
of $P(u)$ and $P'(u)$. It is a standard result of commutative algebra that
the Sylvester matrix of $P(u)$ and $Q(u)$ is singular if and only if the
polynomials $P$ and $Q$ have a common
root. In our case $P(u)$ and $P'(u)$ having a common root is equivalent to
$P(u)$ having a root of higher multiplicity. In the case where $P$ has a
multiple root the a pair of branch points degenerate to a pole and the
genus of the surface decreases by one. We will later work an example where this
occurs.

For a given polynomial it is rather straightforward to work these out,
particularly with the aid of computer algebra systems. In this paper we
did many of the more laborious calculations with Mathematica.\cite{Mathematica}
Some examples are presented in the next section.

\section{Examples}
\subsection{The Korteweg -de Vries Equation (KdV-1)}
The Korteweg-Devries (KdV) equation
\[
u_t + u_{xxx} + (u^2)_x =0
\]
is, of course, completely integrable. The spectrum of the linearized flow can
in principle be understood by the machinery of the inverse scattering
transform, in particular by the construction via Baker-Akheizer
functions detailed in the text of Belokolos, Bobenko, Enolskii, Its
and  Matveev\cite{BBEIM}).
Nevertheless this problem provides a good test for our methods, which
we believe to be considerably simpler and easier to calculate than the
algebro-geometric approach.

Assuming a  traveling wave $u(x-ct)$ the ordinary differential equation
integrates up to
\begin{align}
& - c u_x + u_{xxx} + (u^2)_x =0 \\
& u_{xx}  = a + c u - u^2 = -V'(u)\\
\frac{u_x^2}{2} &= E + a u + c\frac{u^2}{2} - \frac{u^3}{3}=: E - V(u)
\end{align}
where $V(u)=V(u;a,c)$ denotes the effective potential for the Hamiltonian system.
A fundamental quantity is the discriminant of the polynomial $R(u;a,E,c):=E-V(u;a,c)$,
which is given by
\[
{\rm disc}(R(u;a,E,c)) = \frac{1}{12} \left(16 a^3+3 a^2 c^2-36 E a c-6 E c^3-36
   E^2\right).
\]
Notice the KdV equation has periodic solutions if and only if
${\rm disc}(R(u;a,E,c))$ is positive.  Moreover, by scaling (and possibly a map $u \mapsto -u$)
the wave speed $c$ can be assumed to be $c=+1$.

In this case, the Picard-Fuchs system is the following set of five linear equations:
\[
\left(
  \begin{array}{ccccc}
    E & a & c/2 & 1/3 & 0 \\
    0 & E & a & c/2 & 1/3 \\
    a & c & -1 & 0 & 0 \\
    0 & a & c & -1 & 0 \\
    0 & 0 & a & c & -1 \\
  \end{array}
\right)
\left(
  \begin{array}{c}
    I_0 \\
    I_1 \\
    I_2 \\
    I_3 \\
    I_4 \\
  \end{array}
\right)
=
\left(
  \begin{array}{c}
    T \\
    M \\
    0 \\
    2T \\
    4M \\
  \end{array}
\right),
\]
where
\[
\mu_k=\oint\frac{u^k~du}{\sqrt{2R(u;a,E,c)}}~~{\rm and}~~
I_k=\oint\frac{u^k~du}{\left(2R(u;a,E,c)\right)^{3/2}}.
\]
Solving this system implies the various Jacobians arising in Theorem \ref{thm:index} and the modulational
stability index \eqref{MI} can be expressed in terms of the period
$T$ and the mass $M$ as follows:

\begin{align*}
T_E &= \frac{\left(4 a + c^2\right)M + \left(6E+ac\right)T}{12 {\rm disc}(R(u,a,E,c))} \\
\{T,M\}_{a,E} 
&=-\frac{T^2 V'(\frac{M}{T})}{12 ~{\rm disc}(R(u,a,E,c))}\\
\{T,M,P\}_{a,E,c} 
&= \frac{ T^3 (E - V(\frac{M}{T}))}{2 ~{\rm disc}(R(u,a,E,c))}\\
2\Delta_{MI} 
&= \frac{\left(\alpha_{3,0}T^3 +\alpha_{2,1}T^2 M + \alpha_{1,2}T M^2 + \alpha_{0,3}M^3\right)^2}{2^{10}3^7~{\rm disc}^3(R(u,a,E,c))}
\end{align*}
where
\begin{align*}
\alpha_{3,0} &= 36 E + 18 a E c - 8 a^3 \\
\alpha_{2,1} &= 18 E c^2 - 6 a^2 c + 36 a E \\
\alpha_{1,2} &= (-18 c E + 24 a^2 + 3 a c^2 \\
\alpha_{0,3} &= c^3 + 6 a c + 12 E
\end{align*}
These quantities are all positive. The positivity of $T_E$ follows from the result of Schaaf\cite{schaaf} mentioned earlier. The non-negativity of
$\Delta_{MI}$ is clear: in principle the cubic polynomial in the numerator
could vanish but numerics shows that it does not in the region where ${\rm disc}(R(u,a,E,c))>0.$ The positivity of $\{T,M,P\}_{a,E,c}$ is clear. Finally $\{T,M\}_{a,E}$ is positive from Jensen's inequality since
\[
\oint \frac{V'(u)}{\sqrt{R(u;a,E,c)}} =0.
\]

\begin{remark}
Recall from \cite{BrJ} that the Jacobian $\{T,M,P\}_{a,E,c}$ arises naturally as an orientation index
for the gKdV linearized spectral problem for a sufficiently smooth
nonlinearity.  Indeed, one has that $\{T,M,P\}_{a,E,c}<0$ is sufficient
to imply the existence of a non-zero real periodic eigenvalue of the linearized operator
${\bf J}\mathcal{L}$, i.e. an unstable real eigenvalue in $L^2_{\rm per}[T/2,T/2]$.
Moreover, from \cite{MJ1} it follows that if $T_E>0$, then such an eigenvalue
can not exist if $\{T,M,P\}_{a,E,c}$ is positive: however, no such claim can be made in the case
where $T_E<0$.
\end{remark}

Theorem \ref{thm:index} now implies the following index result: if one considers the
linearized operator acting on $L^2({\mathbb T}_k)$ for $k \in {\mathbb N}$
then $u_x$ has $2k$ roots in ${\mathbb T}_k$ and
the number of real eigenvalues, complex eigenvalues, and imaginary
eigenvalues of negative Krein signature satisfy
\begin{equation}
k_{\mathbb I}^- + k_{\mathbb R} + k_{\mathbb C} = 2(k-1)\label{count:kdv}
\end{equation}
In particular when $k=1$, so one is considering stability to perturbations
of the same period, the only eigenvalues lie on the imaginary
axis and have positive Krein signature thus proving orbital
stability of such solutions in $L^2({\mathbb T}_k)$. Furthermore, considered as an
operator on $L^2[-\infty,\infty]$ the spectrum in a neighborhood
of the origin in the spectral domain consists of the imaginary axis
with multiplicity three, thus implying modulational stability of the periodic
traveling wave solutions of the KdV equation.  

\subsection{Example: Modified Korteweg- de Vries (KdV-2)}

The MKdV equation
\[
u_t + u_{xxx} \pm (u^3)_x = 0
\]
arises as a model for wave propagation in plasmas and
as a model for the propagation of interfacial waves in a
stratified medium. It is also integrable and the same caveats apply as for
the KdV regarding the algebro-geometric construction of the spectrum of the
linearized operator.
The MKdV is invariant under the scaling
$x \mapsto \alpha x,t \mapsto \alpha^3 t, u \mapsto \alpha^{-\frac{2}{3}} u,$
and thus the wavespeed $c$ can be scaled to be $c = 0,\pm 1.$ The most
physically and mathematically interesting case is the focusing MKdV (the plus
sign above) with rightmoving waves where $c$ can be scaled to $+1$.
If we scale
the focusing MKdV equation such that $c=1$ the $a,E$ parameter space contains
the familiar swallow-tail fold: see Figure \ref{fig:swallow}.

\begin{figure}
\centering
\includegraphics[scale=.32]{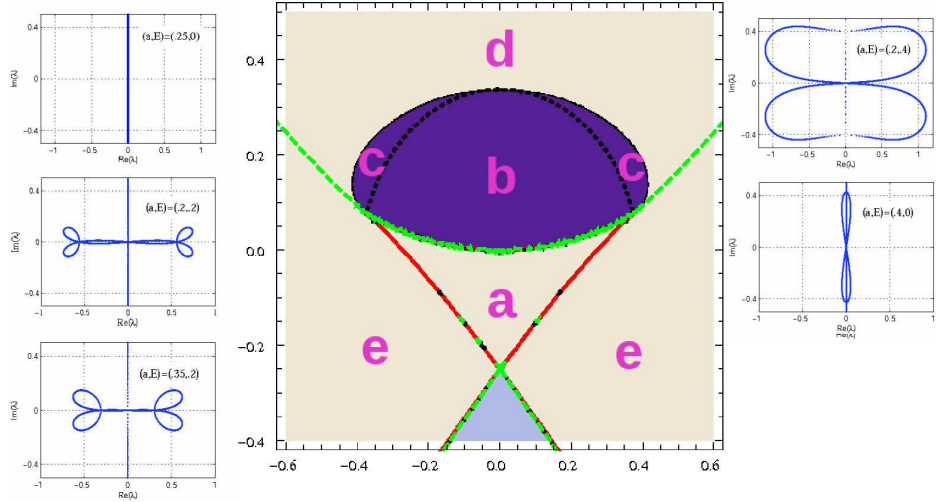}
\caption{The configuration space for focusing MKdV with $c=+1$. The swallowtail
figure divides the plane into regions containing $0,1,$ and $2$ periodic
solutions. The domain is colored according to the sign of $\{T,M,P\}_{a,E,c}$.}
\label{fig:swallow}
\end{figure}

For the focusing MKdV, the swallowtail curve is defined implicitly by the equation
\[
{\rm disc}(E + a u + u^2/2 - u^4/4) = 2 a^2 - 27 a^4 - 4 E + 72 a^2 E - 32 E^2 - 64 E^3 =0
\]
or by the polynomial parametric representation
\begin{align}
a &= s-s^3 \\
E &= \frac{s^2}{2}-\frac{3s^4}{4}.
\end{align}
The soliton solution corresponds to the origin $a=0,E=0.$ Along the upper
(dashed) branch ($s \in [-\frac{\sqrt{3}}{3},\frac{\sqrt{3}}{3}]$) there
are two solutions: a constant solution
and a solitary waves homoclinic to some (non-zero) constant value. Along the
lower (dotted) branch $s \in [-1,-\frac{\sqrt{3}}{3}] \cup [\frac{\sqrt{3}}{3},1]$ there are again two solutions: a constant solution and a non-constant
solution. Along the remaining portions of the curve there is only a
constant solution.  The Riemann surface associated to
the traveling wave solutions is of genus one (a torus)
\[
\frac{u_x^2}{2} = E + a u + \frac{u^2}{2} - \frac{u^4}{4}
\]
except along the swallowtail curve where the discriminant vanishes. In the
case of vanishing discriminant
the torus ``pinches off'' and degenerates to a cylinder. In this case all of
the elliptic integrals can be evaluated in terms of elementary functions.
Unlike the KdV case, where the only periodic solutions on the curve of
vanishing discriminant are constant, the MKdV admits non-trivial periodic
solutions on the swallowtail curve.

The modulational instability index turns out, in this case, to be
particularly simple. After solving the Picard-Fuchs system
\[
\left(
                 \begin{array}{ccccccc}
                  E & a & \frac{1}{2} & 0 & -\frac{1}{4} & 0 & 0 \\
                  0 & E & a & \frac{1}{2} & 0 & -\frac{1}{4} & 0 \\
                  0 & 0 & E & a & \frac{1}{2} & 0 & -\frac{1}{4} \\
                  a & 1 & 0 & -1 & 0 & 0 & 0 \\
                  0 & a & 1 & 0 & -1 & 0 & 0 \\
                  0 & 0 & a & 1 & 0 & -1 & 0 \\
                  0 & 0 & 0 & a & 1 & 0 & -1
                 \end{array}
                 \right) \left(\begin{array}{c}I_0 \\ I_1 \\ I_2 \\ I_3 \\ I_4 \\ I_5 \\ I_6 \end{array}\right) = \left(\begin{array}{c}T \\ M \\ P \\ 0 \\ 2T \\ 4M \\ 6P \end{array}\right)
\]
one finds the following expressions for the various Jacobians:
\begin{align*}
T_E &= -\frac{(3 a^2  - 16 E^2 - 4 E )T + (9 a^2 - 4 E  - 1)M}{16 ~{\rm disc}(R(u;a,E,1))}\\
\{T,M\}_{a,E} &= -\frac{(3 a^2 - 4 E ) T^2 + (4 E - 1) P T +  P^2}{16 ~{\rm disc}(R(u;a,E,1))}\\
\{T,M,P\}_{a,E,c} &= -\frac{(2a^2 - 4 E )T^3 + 4 E P T^2 -  T P^2 + P^3}{32 ~{\rm disc}(R(u;a,E,1))}\\
\Delta_{MI} &= \frac{\left(\alpha_{0,3}T^3 + \alpha_{1,2}T^2 P + \alpha_{2,1} T P^2 + \alpha_{3,0}P^3\right)^2}{4194304~{\rm disc}(R(u;a,E,1))^3}
\end{align*}
where
\[
R(u;a,E,1):=E + a u + u^2/2 - u^4/4
\]
as before and
\begin{align*}
\alpha_{3,0} &= (1+36 E - 27 a^2) \\
\alpha_{2,1} &= (27 a^2 + 144 E^2 - 60 E) \\
\alpha_{1,2} &= (36 E - 240 E^2 - 18 a^2 + 108 E) \\
\alpha_{0,3} &=  (54 a^4-180 a^2 E + 144 E^2 + 64 E).
\end{align*}

We note a few things. First notice that, while the Picard-Fuchs system
involves $T$, $M$, and $P$ the resulting Jacobians only involve $T$ and $P$. While this
is not obvious from the point of view of linear algebra there is a clear
complex analytic reason why this must be so: the Abelian differentials
defining $T$ and $P$ have zero residue about the point at infinity, as
do $T_E,T_a \ldots P_c,$ while $M$ has a non-vanishing residue at
infinity. Thus $T_E,T_a \ldots P_c$ must be expressible in terms of only
$T$ and $P$.

Secondly we note that while there are two distinct families of solutions inside
the swallowtail they have the same orientation index $\{T,M,P\}_{a,E,c}$ and modulational
instability index $\Delta_{MI}$. This is special to MKdV, and follows from the
fact that the integrals over one real cycle can be simply related to the
integrals over the other real cycle via
\begin{align}
\oint_{\gamma_{\rm left}} \frac{du}{\sqrt{R(u;a,E,c)}} &= \oint_{\gamma_{\rm right}} \frac{du}{\sqrt{R(u;a,E,c)}} \\
\oint_{\gamma_{\rm left}} \frac{u ~du}{\sqrt{R(u;a,E,c)}} &= \oint_{\gamma_{\rm right}} \frac{u~ du}{\sqrt{R(u;a,E,c)}} + 2 \sqrt{2}\pi\\
\oint_{\gamma_{\rm left}} \frac{u^2~ du}{\sqrt{R(u;a,E,c)}} &= \oint_{\gamma_{\rm right}} \frac{u^2~du}{\sqrt{R(u;a,E,c)}}
\end{align}
by deforming the contour onto the other cycle and picking up the contribution
from the residue at infinity (which vanishes for $T$ and $P$ and is one for $M$).
Since the orientation and modulational instability indices are built of derivatives
of the above quantities these indices must be the same for both families
of solutions. This observation extends
the calculation of Haragus and Kapitula \cite{HK} for the zero amplitude
waves to the periodic waves on the swallowtail curve.

As in the KdV case the modulational instability index, which is a homogeneous
polynomial of degree $6$ in $T$ and $P$, can be expressed as the square of a
homogeneous polynomial of degree $3$ over an odd power of the discriminant
of the polynomial $R(u;a,E,c)$.  A similar expression holds in the defocusing case, as well
as for general values of $c$.  The sign of this quantity
is obviously the 
same as of the sign of the discriminant
of the quartic, which is in turn 
positive if the quartic has no real roots
or 4 real roots, and 
negative if the quartic has two only real roots. Thus we establish the following surprising fact:

\begin{thm}\label{thm:mkdv}
The traveling wave solutions to the MKdV equation
\[
u_t + u_{xxx} \pm (u^3)_x
\]
are modulationally unstable for a given set of
parameter values if the polynomial
\[
E + a u + c u^2/2 \pm u^4/4
\]
has two real roots, and is modulationally stable if it has four real roots.
\end{thm}

\begin{remark}
In the case of focusing MKdV, Theorem \ref{thm:mkdv} implies
that if the parameter values give rise to one periodic solution
then this solution is unstable to perturbations of sufficiently
long wavelength. If there are two periodic solutions then the
spectrum of the linearization about one of these solutions in the
neighborhood of the origin consists of the imaginary axis
with multiplicity three. For the case of defocusing
MKdV the situation is reversed: for a given set of parameter values there
can be at most one periodic solution, which has no spectrum off of the imaginary
axis in the neighborhood of the origin.
\end{remark}

Note that while this problem is in principle completely solvable using
algebro-geometric techniques, Theorem  \ref{thm:mkdv} appears to be new.
We suspect this is because the classical algebro-geometric calculations are
sufficiently complicated that they are difficult to do in general. For
examples of this sort of calculation see the original text of
Belokolos et. al.\cite{BBEIM} as well as the papers of Bottman and Deconinck\cite{BD}
and Deconinck and Kapitula\cite{DK}.

We now summarize the more interesting situation of the focusing MKdV 
in Figure \ref{fig:swallow} and below:
\begin{itemize}
\item[(a)] There are two families of solutions in this region. For both
of these solutions the modulational instability index is positive and
thus in a neighborhood of the origin the imaginary axis is in the
spectrum with multiplicity three.
Solutions in this region have $T_E>0$, $\{T,M\}_{a,E}>0$, and $\{T,M,P\}_{a,E,c}>0$
implying $k_\mathbb{I}^-+k_\RM + k_\CM= 2(k-1).$

\item The solutions in the remaining regions have a modulational instability index that is negative showing that they are always unstable to perturbations of sufficiently long wavelength.

\item[(b)] In this region $T_E<0$, $\{T,M\}_{a,E}<0$, and $\{T,M,P\}_{a,E,c}<0$ implying
  $k_\mathbb{I}^-+k_\RM + k_\CM= 2k-1$.  

\item[(c)] In this region $T_E<0$, $\{T,M\}_{a,E}>0$, and $\{T,M,P\}_{a,E,c}<0$  implying
  $k_\mathbb{I}^-+k_\RM + k_\CM= 2k-1.$ 
As one crosses between regions $b$ and $c$ the indices
$n(\mathcal L)|_{H^1})$ and $n(D)$ both increase (resp. decrease) by one, leaving the total count the same.

\item[(d)] In this region $T_E<0$, $\{T,M\}_{a,E}>0$, and $\{T,M,P\}_{a,E,c}>0$  implying
  $k_\mathbb{I}^-+k_\RM + k_\CM= 2(k-1).$

\item[(e)] In this region  $T_E>0$, $\{T,M\}_{a,E}>0$, and $\{T,M,P\}_{a,E,c}>0$  implying
  $k_\mathbb{I}^-+k_\RM + k_\CM= 2(k-1).$
\end{itemize}
In regions $(a)$, $(d)$ and $(e)$ when considering periodic perturbations ($k=1$)
one finds that $k_\mathbb{I}^-+k_\RM + k_\CM=0$ implying both spectral and orbital stability
in $L^2({\mathbb T}_1)$.  However, as mentioned above, in regions $(d)$ and
$(e)$ the solution is spectrally unstable in $L^2({\mathbb T}_k)$
for $k\in\mathbb{N}$ sufficiently large.  Moreover, in regions $b$ and $c$
there always exists a non-zero real periodic eigenvalue,
i.e. the linearized operator ${\bf J}\mathcal{L}$ acting on $L^2({\mathbb T}_k)$
always has a non-zero real eigenvalue and hence such solutions are always spectrally
unstable.

\begin{remark}
It should be noted that the above counts are consistent with the calculations of 
Deconinck and Kapitula \cite{DK} in which they consider stability of the cnoidal wave solutions
\[
U(x,t;\kappa)=\sqrt{2}\mu\cn\left(\mu x-\mu^2(2-\kappa^2)t;k\right)
\]
of the focusing MKdV equation, where $\mu>0$ and $\kappa\in[0,1)$.  Such solutions always correspond to regions
$(b)$ and $(d)$ along with the constraint $a=0$.  There, the authors find numerically
that there is a critical elliptic-modulus $\kappa^*\approx 0.909$ such that solutions
with $0\leq \kappa<\kappa^*$, corresponding to region $(b)$ are orbitally stable in $L^2({\mathbb T}_k)$
while solutions with $\kappa^*<\kappa<1$, corresponding to region $(d)$ 
are spectrally unstable in $L^2({\mathbb T}_k)$ for all $k\in\mathbb{N}$ 
due to the presence of a non-zero real eigenvalue of the linearized operator.  

\end{remark}

\subsection{Example:  $L^2$  critical Korteweg-de Vries (KdV -4)}

Finally we consider the equation
\[
u_t + u_{xxx} + (u^5)_x = 0.
\]
This equation is not a physical model for any system that we are
aware of but is mathematically interesting for a number of reasons. This is
the power where
the solitary waves first go unstable. Equivalently this is the $L^2$
critical case, where one has the scaling $u \mapsto \sqrt{\gamma} u(\gamma x)$
preserving the $L^2$ norm and the relative contributions of the kinetic
and potential energy to the Hamiltonian.  Again we focus on the focusing case,
which is the more interesting, and we scale everything so that $c=+1$.

Again the parameter space is divided by a swallowtail curve into
regions containing no periodic solution, one periodic solution, and
two periodic solutions. The implicit representation is given by
\[
-48 a^2 + 3125 a^6 + 96 E -11250 a^4 E + 10800a^2 E^2 - 1728 E^3 + 7776 E^5=0
\]
or parametrically by
\begin{align}
a &= s^5-s \\
E &= \frac{s^2}{2} - \frac{5 s^6}{6}.
\end{align}
Again the picture is qualitatively similar to the MKdV case: the
portion of the swallowtail parameterized by $s \in (-5^{-\frac{1}{4}},5^{-\frac{1}{4}})$
represents parameter values for which there are two solutions: one constant
and one homoclinic
to a constant, with the origin representing the soliton solution (the solution
homoclinic to zero) and the zero solution. The portions of the curves
parameterized by $s \in (-\frac{1}{\sqrt[4]{5}},-1)\cup(\frac{1}{\sqrt[4]{5}},1)$ represent
parameter values for which there are two solutions: a constant and a periodic
solution. The remainder of the curve represents parameter values for which
there is only the constant solution.

The Picard-Fuchs system is following set of eleven equations:
\[
\left(
   \begin{array}{ccccccccccc}
    E & a & \frac{1}{2} & 0 & 0 & 0 & -\frac{1}{6} & 0 & 0 & 0 & 0 \\
    0 & E & a & \frac{1}{2} & 0 & 0 & 0 & -\frac{1}{6} & 0 & 0 & 0 \\
    0 & 0 & E & a & \frac{1}{2} & 0 & 0 & 0 & -\frac{1}{6} & 0 & 0 \\
    0 & 0 & 0 & E & a & \frac{1}{2} & 0 & 0 & 0 & -\frac{1}{6} & 0 \\
    0 & 0 & 0 & 0 & E & a & \frac{1}{2} & 0 & 0 & 0 & -\frac{1}{6} \\
    a & 1 & 0 & 0 & 0 & -1 & 0 & 0 & 0 & 0 & 0 \\
    0 & a & 1 & 0 & 0 & 0 & -1 & 0 & 0 & 0 & 0 \\
    0 & 0 & a & 1 & 0 & 0 & 0 & -1 & 0 & 0 & 0 \\
    0 & 0 & 0 & a & 1 & 0 & 0 & 0 & -1 & 0 & 0 \\
    0 & 0 & 0 & 0 & a & 1 & 0 & 0 & 0 & -1 & 0 \\
    0 & 0 & 0 & 0 & 0 & a & 1 & 0 & 0 & 0 & -1
   \end{array}
   \right)\left(\begin{array}{c}I_0 \\I_1 \\ I_2 \\ I_3 \\ I_4 \\ I_5 \\ I_6 \\ I_7 \\ I_8 \\ I_9 \\ I_{10} \end{array}\right) = \left(\begin{array}{c}\mu_0 \\ \mu_1 \\ \mu_2 \\ \mu_3 \\ \mu_4 \\ 0 \\ 2\mu_0 \\ 4\mu_1 \\ 6\mu_2 \\ 8\mu_3 \\ 10\mu_4\end{array}\right).
\]
We have explicit expressions for the various Jacobians arising in the theory,
but they are cumbersome and we will no reproduce them here:
it should be noted they are homogeneous polynomials in $\mu_0 (=T)$, $\mu_1 (=M)$, $\mu_3$, and $\mu_4$.
Moreover, we note that these quantities turn out to be independent of $P=\mu_2$
since the differential corresponding to momentum has a non-trivial residue at infinity,
similar to the case of the MKdV.

\begin{figure}
\centering
\includegraphics[scale=.35]{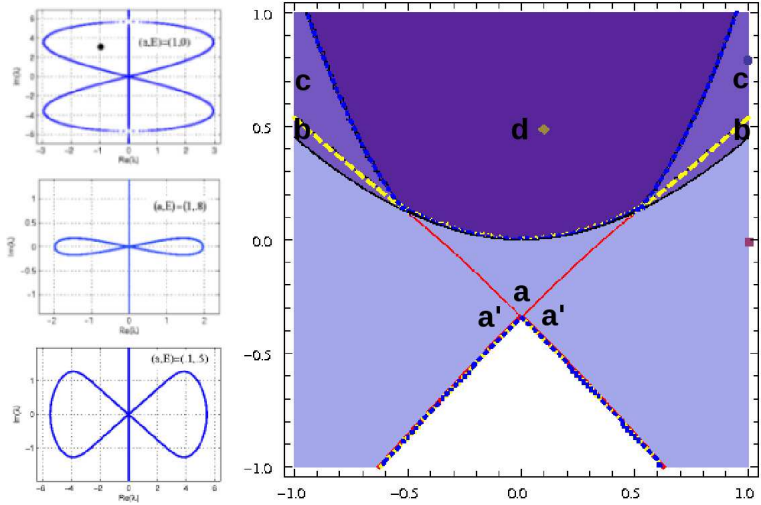}
\caption{The configuration space for focusing KdV-4 with $c=+1$. The
swallowtail figure divides the plane into regions containing $0,1,$ and $2$
periodic solutions.}
\label{fig:swallow4}
\end{figure}

The stability diagram of these solutions is depicted in figure \ref{fig:swallow4}.
Numerics indicate that the modulational instability index is always
negative, indicating that solutions are always modulationally
unstable. There are three curves emerging from the cusps of the swallowtail.
The lowest of these is the curve on which the orientation index $\{T,M,P\}_{a,E,c}$ vanishes,
the middle (dashed) where $T_E$ vanishes and the upper (dotted) where $\{T,M\}_{a,E}$ vanishes.

The behavior in the various regions is summarized as follows:
\begin{itemize}
\item[(a)]There are two solution families in this region, both of
which satisfy $T_E>0$, $\{T,M\}_{a,E}>0$, and $\{T,M,P\}_{a,E,c}>0$.
This implies that there are no real periodic eigenvalues and that $k_{\mathbb I}^- + k_{\mathbb R} + k_{\mathbb C}=2(k-1)$
\item [(a')] There is only one solution family in this region, otherwise the
behavior is the same as in region (a')
\item [(b)] The family of solutions in this region has $T_E>0$, $\{T,M\}_{a,E}>0$, and $\{T,M,P\}_{a,E,c}<0$.
This implies that $k_{\mathbb I}^- + k_{\mathbb R} + k_{\mathbb C}=2k-1$
\item[(c)] The family of solutions in this region has $T_E<0$, $\{T,M\}_{a,E}>0$, and $\{T,M,P\}_{a,E,c}<0$.  This implies that $k_{\mathbb I}^- + k_{\mathbb R} + k_{\mathbb C}=2k-1.$
\item[(d)] The family of solutions in this region has $T_E<0$, $\{T,M\}_{a,E}<0$, and $\{T,M,P\}_{a,E,c}<0$.  This implies that $k_{\mathbb I}^- + k_{\mathbb R} + k_{\mathbb C}=2k-1.$
\end{itemize}
It follows that solutions in region $(a)$ are orbitally stable in $L^2({\mathbb T}_k)$
and spectrally unstable in $L^2({\mathbb T}_k)$ for $k\in\mathbb{N}$ 
sufficiently large.
Moreover, solutions in the remaining regions are spectrally unstable
in $L^2({\mathbb T}_k)$ for any $k\in\mathbb{N}$ due to the presence
of a non-zero real periodic eigenvalue.

\section{Conclusions}
We have proven an index theorem for the linearization of Korteweg-de Vries
type flows around a traveling wave solution and shown that the number of
eigenvalues in the right half-plane plus the number of purely imaginary eigenvalues
of negative Krein signature given be expressed in terms of the
Hessian of the classical action of the traveling wave ordinary differential
equation or (equivalently) in terms of the Jacobian of the map from the
Lagrange multipliers to the conserved quantities. In the case of polynomial
nonlinearity these quantities can be expressed in terms of homogeneous
polynomials in Abelian integrals on a finite genus Riemann surface.

The main drawback of the result is that it does not really
distinguish between the eigenvalues in the right half-plane,
which lead to an instability, and the imaginary eiegnvalues
of negative Krein signature, which are generally not expected to
lead to an instability. The index does distinguish between real
eigenvalues and imaginary eigenvalues of negative Krein signature,
but only modulo two. This is sufficient to deal with the solitary
wave case, where the only possible instability mechanism is the
emergence of a single real eigenvalue from the origin. However
in the periodic problem, where the behavior of the spectral problem is much
richer, it would be preferable to have more information.

We believe that a stronger result is possible: namely that there is
spectrum off of the imaginary axis if and only if the modulational
instability index is negative. It follows from the results of this paper
and the previous work of Bronski and Johnson that negativity of the
modulational instability index is a sufficient condition for instability.
Showing necessity would amount to showing that existence of any spectrum
off of the imaginary axis would imply the existence of spectrum off
of the imaginary axis in a neighborhood of the origin. In numerical
experiments that we have conducted this has always been true.

{\bf Acknowledgements:} JCB gratefully acknowledges support from the National Science Foundation under NSF grant DMS-0807584. MJ gratefully acknowledges 
support from a National Science Foundation Postdoctoral Fellowship. TK gratefully acknowledges the support of a Calvin Research Fellowship and the National Science Foundation under grant DMS-0806636.
\bibliography{GKdVbib}

\end{document}